\documentclass[12pt,reqno]{amsart}
\usepackage{graphicx}
\usepackage{tikz-cd}
\usepackage{amsmath}
\usepackage{amssymb}
\usepackage{amscd}
\usepackage{mathrsfs}
\usepackage[all,cmtip]{xy}
 \usepackage{color}
\usepackage{amsthm}
\usepackage{caption}
\usepackage{mathtools}

\calclayout

\theoremstyle{definition}

\theoremstyle{definition}

\theoremstyle{definition}
\newtheorem{remark}{Remark}
\theoremstyle{plain}
\newtheorem{theorem}{Theorem}
\newtheorem{proposition}[theorem]{Proposition}
\newtheorem{corollary}[theorem]{Corollary}
\theoremstyle{plain}

\newcommand\com[1]{}

\newcommand\Cc{{\let\mathcal\mathscr\mathcal C}}

\newcommand\D{{\mathcal D}}
\newcommand\E{\mathcal{E}}
\newcommand\ED{\tilde{\mathcal E}}

\newcommand\g{{\frak g}}

\newcommand{\HK}{\mathcal{H}}
\newcommand{\HD}{\tilde{\mathcal{H}}}

\renewcommand\l{\lambda}

\newcommand\Ll{{\let\mathcal\mathscr\mathcal L}}

\newcommand\op[1]{\mathop{\rm #1}\nolimits}

\newcommand\p{\partial}

\newcommand\R{{\mathbb R}}

\newcommand{\sym}{\mathfrak{g}}
\newcommand{\Sym}{\mathcal{G}}

\begin{document}

 \title[Differential invariants of Kundt spacetimes]{Differential invariants of Kundt spacetimes}
 \author[B. Kruglikov]{Boris Kruglikov$^\dagger$}
\address{\hspace{-17pt}
$^\dagger$%
Department of Mathematics and Statistics, UiT the Arctic University of Norway, 9037 Troms\o\, Norway.\quad
E-mail: {\tt boris.kruglikov@uit.no}.}
\address{\hspace{-17pt}
$^\dagger$%
Department of Mathematics and Natural Sciences, University of Stavanger, 4036 Stavanger, Norway.}
 \author[E. Schneider]{Eivind Schneider$^\ddagger$}
\address{\hspace{-17pt}
$^\ddagger$%
Faculty of Science, University of Hradec Kr\'alov\'e, Rokitansk\'eho 62, Hradec Kr\'alov\'e
500 03, Czech Republic.\quad
E-mail: {\tt eivind.schneider@uhk.cz}. }
 \date{}

 \vspace{-14.5pt}
 \begin{abstract}
We find generators for the algebra of rational differential invariants for
general and degenerate Kundt spacetimes and relate this to other approaches to the equivalence problem for Lorentzian metrics. Special attention is given to dimensions three and four. 
 \end{abstract}

 \maketitle %\tableofcontents

\section{Introduction}

There are various approaches for distinguishing and classifying
pseudo-Riemannian metrics, and to determine their Killing vectors,
important in mathematical relativity. 

 \begin{description}
\item[spi]
Scalar polynomial invariants are obtained by complete contractions of
the Riemann tensor, its covariant derivatives and their tensor products.
\item[cci]
 Cartan curvature invariants are obtained from structure functions of the absolute parallelism on the reduced frame bundle, and their derivatives.
\item[sdi]
Scalar differential invariants are obtained as the invariants of
the diffeomorphism pseudogroup acting in the space of jets of metrics.
 \end{description}

By a theorem of Weyl \cite{W}, spi are sufficient to distinguish
Riemannian and % as well as
generic pseudo-Riemannian metrics. 
However, there exist non-equivalent metrics of non-positive signature
with the same spi. For instance, VSI spaces with vanishing scalar
(polynomial) invariants \cite{PPCM} are indistinguishable
from the Minkowski spacetime by spi,
and this is not related to any symmetry of the problem
\cite{Kou}. In addition, a sufficient number of spi has never been
specified in the literature\footnote{Even to specify all invariants
of the second order in dimension 4 required some efforts
 % actually: traces of powers of Ricci and Weyl tensors as operators
\cite{ZM}.}.
For instance, while principally
known to resolve the count of Killing vectors for Riemannian metrics
\cite{Ker,S,CO}, the number and complexity of the involved spi is
beyond a reasonable computational capacity \cite{KT}.

Cartan invariants, on the other hand, are universally applicable in the
study of spacetimes in general relativity, since they do separate metrics.
They are especially popular in the form of the Cartan-Karlhede algorithm,
in the Penrose-Newman formalism, etc \cite{Kar,PR,ESEFE}.
In the original approach, the invariants live on the
Cartan frame bundle \cite{C}, and they classically correspond to covariants.
Those combinations that are invariant with respect to the structure group are
actually cci, and can be treated on the base manifold $M$.
They are obtained as components of the curvature tensor
and its covariant derivatives by normalizations of the group parameters.
The invariants are usually considered local and smooth.

Scalar differential invariants, though a classical tool in differential geometry
\cite{T}, have never been pivotal in relativity applications.
An essential difference between cci and sdi is the following.
With the approach of Élie Cartan, the frame bundle depends on the metric
and the remaining freedom is a finite-dimensional structure group
(a subgroup of the pseudo-orthogonal group). With the approach of
Sophus Lie, an infinite-dimensional transformation pseudogroup
(a subgroup of the diffeomorphism group) acts on the space of jets
of all metrics in the class and sdi are invariant functions of this action. 
%If interpreted correctly, both spi and cci are sdi, albeit special ones. Both types of invariants are elements of the differential algebra of scalar differential invariants. In this sense, the theory of sdi provides a unifying framework. It also gives more flexibility regarding the choice of invariants used for distinguishing metrics.
With mild assumptions, the differential invariants can be assumed rational
in jet-variables \cite{KL2} and global\footnote{This means they
are only subjects to nonequalities but not to inequalities (which can
happen for smooth cci). Contrary to spi, defined for all metrics,
both cci and sdi have a domain of definition.} since 
they do separate generic orbits of the action.

In this paper, we follow the latter (sdi) approach and apply it to Lorentzian metrics that are indistinguishable by spi. These metrics are known to be contained in the
Kundt class \cite{Ku,CHPP}. Furthermore, it was shown in \cite{CHP}
(in dimensions 3 and 4) that a Lorentzian spacetime is either
weakly $I$-nondegenerate (a discrete point in the set of metrics
with the same spi; thus locally characterized by them) or is
a degenerate Kundt spacetime (this class contains the VSI spacetimes).
For related results in higher dimensions see \cite{CHP2}.

Our main goal is to distinguish degenerate Kundt spacetimes by describing the
algebra $\mathcal{A}$ of rational sdi following the theory from \cite{KL2}.  In particular, we will find a set of generators for $\mathcal{A}$. 
We do it in adapted coordinates, though the results are applicable
to Kundt metrics (general or degenerate) in any coordinate system.
We also find an invariant frame adapted to Kundt spacetimes,
thus relating to the Cartan approach. Several subclasses of Kundt metrics
have been discussed in the literature using the Cartan-Karlhede algorithm,
see \cite{MMC} and references therein. For Kundt waves in 4D we compared
both approaches, via sdi and cci, in \cite{KMS}. In the present paper
we discuss the general situation in general dimension $n$, with special attention given to dimensions 3 and 4.

The structure of the paper is as follows. In Section \ref{S2}, we introduce the necessary concepts and notations regarding Kundt spacetimes and their jets, and recall the Lie-Tresse approach to differential invariants. For Kundt spacetimes (both general and degenerate), we choose adapted coordinates that result in metric tensors of a particular shape, and we write down the Lie
pseudogroup of transformations preserving this shape. In Section \ref{S3}, we count the number of algebraically independent sdi depending on the jet-order, and provide Hilbert functions and Poincaré functions for the algebras of differential invariants.
Afterwards, in Section \ref{S4}, we find generators for the algebra of
differential invariants in general dimension. We finish by presenting a simplified version of
invariants for three- and four-dimensional Kundt spacetimes. In particular, we write down a generating set of invariants in coordinates in dimension 3. 
In the appendix we demonstrate how the class of degenerate Kundt 
spacetimes arises by consideration of relative invariants of the Lie pseudogroup of shape preserving transformations.

\section{Setup: Jets and Pseudogroups} \label{S2}

We first review the general theory and then apply it to Kundt spacetimes.

\subsection{Jets and Differential invariants}\label{S21}

The notion of jet-space formalizes the computational devise of truncated
Taylor polynomials; we refer for details to \cite{KL1}.
If $x^i$ are coordinates on $X=\R^n$ and $y^j$ are coordinates on $Y=\R^m$,
then the jet-space $J^k=J^k(X,Y)$ of $k$-jet of maps from $X$ to $Y$
has coordinates $y^j_\sigma$ for multi-indices $\sigma=(i_1,\dots,i_n)$,
$i_s\ge0$, $|\sigma|=\sum i_s\leq k$. The same applies for general $X,Y$
with local coordinates $x^i,y^j$, called independent and dependent variables,
respectively. Any map $\psi:X\to Y$, $x^i\mapsto y^j=y^j(x)$ lifts to the map
$j_k\psi:X\to J^k$ given by $x^i\mapsto y^j_\sigma=\partial y^j(x)/\p x^\sigma$.

The jet-space $J^k$ is equipped with the Cartan distribution $\Cc_k\subset TJ^k$,
where for a point $a_k\in J^k$ the space $\Cc_k(a_k)$ is spanned by
all $n$-planes $T_{a_k}j_k\psi(X)$,  $j_k\psi(X)\ni a_k$. A differential equation of order $\leq l$
can be geometrically interpreted as a submanifold $\E^l\subset J^l$,
and its solutions are integral manifolds of $\Cc_l$. 

By differentiating the defining equations for $\E^l$, we obtain the prolonged equations $\E^{l+i}$ for $i\geq 1$. A smooth solution of $\E^l$ is also a smooth solution of $\E^{l+i}$. For jet spaces, we have the projections  $\pi_{j,i} \colon J^j \to J^i$ for $i<j$, and we define $\E^k= \pi_{l,k}(\E^l)\subset J^k$ for $k=0,\dots, l-1$.  

A transformation group $G$ on $J^0=X\times Y$ or
a local Lie pseudogroup $G\subset\op{Diff}_\text{loc}(X\times Y)$
canonically lifts to a pseudogroup acting in $J^k$ by the condition
that the class of integral manifolds $j_k\psi(X)$ (or equivalently
the Cartan distribution $\Cc_k$) is preserved.
For $G\ni g:J^0\to J^0$ the prolongation is denoted by
$g^{(k)}:J^k\to J^k$. If $G$ consists of symmetries of a PDE $\E$,
then there is an induced action $g^{(k)}:\E^k\to\E^k$.

Differential invariants of order $\leq k$ are functions $f:\E^k \subset J^k\to\R$
that are constant on the orbits of the prolonged action: $f\circ g^{(k)}=f$
(when there is no PDE, one considers functions on $J^k$ that are constant on orbits).
The spaces of all such invariants $\mathcal{A}_k$ unite over $k$
into the algebra of differential invariants
 $$
\mathcal{A}=\lim\limits_{k\to\infty}\mathcal{A}_k.
 $$
One also exploits invariant derivations $\nabla:\mathcal{A}_k\to\mathcal{A}_{k+1}$
of this algebra, which are invariant horizontal vector fields. In other words, they are invariant operators of the form $a^i \D_{x^i}$, where $a^i$ are functions on $\E^r \subset J^r$ for some $r$ and $\D_{x^i}$ are total derivative operators: $\D_{x^i} (f)|_{j^{k+1} \psi} = \p_{x^i} (f \circ j^k \psi)$ for every $f \in C^{\infty}(J^k)$ and $\psi\colon X \to Y$. 

By \cite{KL2}, under mild assumptions\footnote{The pseudogroup $G$
acts transitively on $J^0$ (this can be further relaxed) and algebraically
on the fibers of the projections $J^k\to J^0$; these assumptions will be
satisfied in the case we study.},
the invariants can be assumed rational in the jet-variables
$y^j_\sigma$, $|\sigma|>0$, and even polynomial in jets of sufficiently
high order $|\sigma|>k_0$.
From now on we will suppose that $\mathcal{A}$ consists of such rational-polynomial
functions. Moreover, the main theorem of \cite{KL2} gives Lie-Tresse type
generation of $\mathcal{A}$ by a finite set of differential invariants $I_a$
and invariant derivations $\nabla_b$, so that generic $G$-orbits in
$\E^\infty \subset J^\infty$ are separated by $I_a$ and their derivatives
 $$
\nabla_BI_a=\nabla_{b_1}\cdots\nabla_{b_t}I_a.
 $$
Here $B$ denotes the multi-index $B=(b_1,\dots,b_t)$.
More precisely, for every $k$ the orbits separated by the invariants
form a Zariski open set in  $\E^k$ (so that the complement intersects any fiber  $\E^k\to J^0$
by a proper algebraic set) and starting from some jet-level $k_1$
no new singularities appear for $k>k_1$.

%Note that when we say that $\mathcal{A}$ is generated by
%$\{I_a\,|\,\,\nabla_b\}$ we should, in general, also include into 
%the set of generators the structure functions $c_{ij}^k$, given by 
In general, the invariant derivations need not commute, meaning that there are nonvanishing structure functions $c_{ij}^k$, given by
 $$
[\nabla_i,\nabla_j]=c_{ij}^k\nabla_k.
 $$
When $\mathcal{A}$ contains $n$ horizontally independent invariants
$I_1,\dots,I_n$, i.e.\ such that their restriction to a generic
holonomic jet-section $j^\infty\psi$ are functionally independent,
the functions $c_{ij}^k$ can be derived from the invariants $\nabla_BI_a$.
%This is indeed the case with Kundt spacetimes, so we will not 
%specifically mention the structure functions henceforth.

With this approach, the equivalence problem for generic solutions $\psi$ of $\E$ with respect to $G$ is solved as follows.
Let us define the signature
 $$
\Psi:X\ni x\mapsto(I_a,\nabla_bI_a)\in\R^q
 $$
as the map with the components consisting of the basic invariants
and their first derivatives ($q$ is the total number)
computed on the submanifold $j_\infty\psi(X)\subset  \E^\infty \subset J^\infty$. Then
$\psi_1$ is $G$-equivalent to $\psi_2$ iff the {\it signature varieties}
$\Psi_1(X),\Psi_2(X)$ coincide, see \cite{C,O} 
(actually this statement can be localized in $X$). In general, the differential algebra of differential invariants is not freely generated; there are differential syzygies. In particular, this implies that not every $n$-dimensional submanifold of $\mathbb R^q$ is a signature variety since the syzygies manifest as differential constraints on the signature varieties.

We note that the definition of signature variety depends on the chosen set of generators, and so does the integer $q$.

\subsection{General Kundt spacetimes}

An $n$-dimensional Lorentzian manifold $(M,g)$ is a Kundt spacetime if it admits a null congruence that is geodesic, expansion-free,
shear-free and twist-free. In other words, there exists a vector field $\ell$ such that\footnote{All
contractions, norms and raising-lowering are with respect to $g$.
We write $\mathbb{D}^g$ for the Levi-Civita connection
to distinguish from invariant derivations $\nabla_i$ 
exploited in generation of the algebra $\mathcal{A}$.}
 % NB: problem of notations: two \nabla !
 $$
\|\ell\|^2=0,\quad \mathbb{D}^g_\ell\ell=0,\quad \op{Tr}(\mathbb{D}^g\ell)=0,\quad
\|\mathbb{D}^g\ell^{\text{sym}}\|^2=0,\quad \|\mathbb{D}^g\ell^{\text{alt}}\|^2=0.
 $$
% \com{
% $$
%\ell^a\ell_a=0,\quad \ell^a_{;a}=0,\quad
%\ell^a\ell_{b;a}=0,\quad \ell^{(a;b)}\ell_{(a;b)}=0,\quad \ell^{[a;b]}\ell_{[a;b]}=0.
% $$} 
 The twist-free condition is equivalent to Frobenius-integrability of $\ell^\perp$.
 Thus we have embedded integrable distributions $\R\cdot\ell\subset\ell^\perp$
of dimension 1 and codimension 1 on $M$. Let $\lambda$ denote the foliation corresponding to $\mathbb R \cdot \ell$, and $\Lambda$ the foliation corresponding to $\ell^\perp$. The (local) quotient of the corresponding foliations
$\bar M=\Lambda/\lambda$ has dimension $n-2$;
the other Kundt conditions translate to the claim that the degenerate symmetric bivector $g|_{\lambda^\perp}$
projects to a Riemannian metric $h=(h_{ij})$ on $\bar{M}$. Below we will denote by $x=(x^1,\dots,x^{n-2})$
both local coordinates on $\bar{M}$ and their pullback on $M$.

This implies the well-known claim \cite{Ku,ESEFE,CHPP} that in some
local coordinates $(u,x,v)$ on $M$ any Kundt metric can be written as follows
 \begin{equation}\label{Kundt}
g= du \left(dv+H(u,x,v)\,du+W_i(u,x,v)\,dx^i\right) + h_{ij}(u,x)\,dx^idx^j.
 \end{equation}
In these coordinates $\ell=\p_v$ and $\ell^\perp=\{du=0\}$.  We let $\bar{M}_u$ denote $\Lambda_u/\lambda$, where $\Lambda_u$ is a leaf of $\Lambda$ (leaves of $\Lambda$ are parametrized by $u$, as they are given by $u=\text{const.}$). 

\subsection{Shape-preserving transformations}

Now we determine the Lie pseudogroup $\Sym$ of diffeomorphisms preserving the class of Kundt metrics given by \eqref{Kundt}. In other words, we find the transformations preserving the shape of such metrics.
In the case $n=4$ this was done in \cite{PPCM}.

 \begin{theorem}\label{T1}
The transformations preserving the shape of \eqref{Kundt} take the form
 \begin{equation}\label{fG}
\Sym\ni\varphi: (u,x^i,v) \mapsto \left(C(u),A^i(u,x),\frac{v}{C'(u)}+B(u,x)\right),\quad
\det[A^i_{x^j}]\neq0, B_u \neq 0.
 \end{equation}
The Lie algebra $\sym$ of this Lie pseudogroup $\Sym$ consists of vector fields of the form
 \begin{equation}\label{fg}
\xi=c(u)\p_u + a^i(u,x)\p_{x^i} + \bigl(b(u,x)-c'(u)v\bigr)\p_v.
 \end{equation}
 % where $a^1,...,a^{n-2},b, c$ are smooth functions.
 \end{theorem}

 \begin{proof}
One approach for proving this is to require that for $\varphi\in\Sym$ the metric $\varphi^*g$ has form \eqref{Kundt}
with some other functions $H,W_i,h_{ij}$ of the same type. This gives a PDE system that is easy to solve.
Another approach is to note that transformations preserving the filtration
$\lambda\subset\lambda^\perp$ have the form: $\varphi(u,x,v)=(U(u),X(u,x),V(u,x,v))$.
Taking into account the condition $(\varphi^*g)(\partial_u,\partial_v)=1$ yields explicit affine behavior of $V$ in $v$.
 \end{proof}

Note that the Lie pseudogroup $\Sym$ has four connected components in the smooth topology, % (separated by the signs of $C_u$ and $\det (A^i_{x^j})$),
but it is Zariski connected (hence $\Sym$ is the Zariski closure of the component given by $C_u>0$, $\det (A^i_{x^j})>0$).
This and the particular form of transformations imply that the condition of the global Lie-Tresse theorem \cite{KL2} are met. Thus, the algebra $\mathcal{A}$ of invariants can be assumed to consist of rational-polynomial functions, since such invariants separate orbits in general position.

\subsection{Lifts and jet-prolongations}
 The image of (\ref{Kundt}) in the bundle $S^2 T^*M$ on $M$ determines a subbundle isomorphic to the trivial bundle
 \begin{equation}\label{eN}
\pi:M\times F\to M,\ \text{ where }\ F\subset \R^N,\
N=n-1+\binom{n-1}{2}=\binom{n}{2},
 \end{equation}
 whose sections are exactly Lorentzian metrics of the form \eqref{Kundt}.
The coordinates on $M$ are independent variables $u,x^i, v$, the coordinates
on $\R^N$ are dependent variables $H,W_i,h_{ij}$ ($1\leq i\leq j\leq n-2$),
and the domain $F\subset \R^N$ is given by the requirement that the symmetric matrix defined by $h_{ij}$ is positive definite. % Riemannian metric

Denote by $J^k\pi$ the $k$-th order jet bundle of $\pi$.
It contains the subbundle of jets of Kundt metrics given by the equation
 $$
\E^1=\{(h_{ij})_v=0\} \subset J^1 \pi 
 $$
as well as its prolongations $\E^k\subset J^k\pi$ for $k>0$, given by the
$(k-1)$ differentiations of the above conditions. We use the notation $\E^0=J^0\pi$.
The infinitely prolonged Kundt equation is $\E^\infty\subset J^\infty\pi$. 

The pseudogroup $\Sym$ (and its Lie algebra $\sym$) have the natural lift to $J^0\pi=M\times F$,
$\Sym\ni\varphi\mapsto\varphi^{(0)}\in\Sym^{(0)}\subset\op{Diff}_\text{loc}(J^0)$,
obtained from the requirement that
 $$
g=du \left( dv+H\,du+W_idx^i\right) + h_{ij}\,dx^idx^j\in\pi^*S^2T^*M
 $$
is invariant with respect to every $\varphi^{(0)}$. %Equivalently this means
%that $\Sym$ preserves the shape of \eqref{Kundt}. 
\begin{remark}
 Now $g$ is interpreted as a horizontal symmetric 2-form on $\pi$. The restriction of $g$ to a section $\psi$ of $\pi$ given by $H=H(u,x,v), W_i=W_i(u,x,v), h_{ij}=h_{ij}(u,x)$ is exactly the metric \eqref{Kundt}. The invariant tensors associated with a metric (the Riemann tensor, Ricci tensor, etc.) can be defined for the horizontal form $g$ so that the restriction of such a tensor to a section of $\pi$ gives exactly the corresponding tensor field associated to the metric \eqref{Kundt}.
\end{remark}

To get the formula for the lift following
the notations of Theorem \ref{T1} denote $A^i_j=\p_{x^j}A^i$ and
let $\check{A}^i_j$ be the inverse matrix. Denote also
$B_j=\p_{x^j}B$ and $\check{C}'=(C')^{-1}$.
Then $\varphi^{(0)}$ maps the fiber as follows:
 \begin{alignat*}{1}
h_{ij} &\mapsto \check{A}^k_i\check{A}^l_jh_{kl},\\
W_i    &\mapsto \check{C}'\check{A}^j_iW_j-\check{A}^j_iB_j-2\check{C}'\check{A}^k_i\check{A}^l_jA^j_uh_{kl},\\
H      &\mapsto \check{C}'^2H+\check{C}'^3C''v-\check{C}'B_u+\check{C}'\check{A}^j_iA^i_uB_j-\check{C}'^2\check{A}^j_iA^i_uW_j
+\check{C}'^2\check{A}^k_i\check{A}^l_jA^i_uA^j_uh_{kl}.
 \end{alignat*}

The lift of vector fields from $\sym$, with $a=a(u,x),b=b(u,x),c=c(u)$, is
(here and in what follows $a^j_i=a^j_{x^i}$, $b_i=b_{x^i}$, etc) such
 \begin{alignat*}{1}
\xi^{(0)} =&\ c\p_u + a^i\p_{x^i} + (b-c'v)\p_v -(a^l_ih_{lj}\p_{h_{ij}} + a_i^l h_{li} \p_{h_{ii}})\\
& -(c'W_i+a^j_iW_j+b_i+2a^j_uh_{ij})\p_{W_i} -(2c'H-c''v+b_u+a^j_uW_j)\p_H.
 \end{alignat*}
These prolong further to transformations $\varphi^{(k)}$ and vector fields $\xi^{(k)}$ on $J^k\pi$. Moreover, by the construction of the lift,
prolongations of the pseudogroup $\Sym$ preserve $\E$.

Note that order $k$ differential invariants of $\Sym$ are rational-polynomial functions $f$ on $\E^k\subset J^k\pi$
that satisfy
 \begin{equation}\label{LieDI}
\Ll_{\xi^{(k)}}f=0\ \forall\xi\in\sym.
 \end{equation}
However, there exist functions that satisfy this system of equations, but are not invariant under the entire Lie pseudogroup $\Sym$. An example of this is mentioned in Section \ref{S44}.
%That this is equivalent to $\varphi^{(k)}f=f\ \forall \varphi\in\Sym$ follows from the above-mentioned Zariski-connectedness of $\Sym$.

\subsection{Degenerate Kundt spacetimes}

Degenerate Kundt metrics are the Kundt metrics that satisfy the following additional conditions:
 \begin{itemize}
\item The Riemann tensor $\text{Riem}$ is aligned and of algebraically special type $II$.% $\Rightarrow (W_i)_{vv}=0$.
\item $\nabla(\text{Riem})$ is aligned and of algebraically special type $II$. % $\Rightarrow  H_{vvv}=0$.
 \end{itemize}
In terms of \eqref{Kundt}, the first condition implies $(W_i)_{vv}=0$ while the second implies $H_{vvv}=0$. It follows that $\nabla^{(k)}(\text{Riem})$ is aligned and of algebraically special type $II$ for every positive integer $k$.  (In all cases we understand the set of type $II$ tensors to also include the more special types $III$, $D$,   etc.) %(Check what follows from \cite{CHPP}.)

In 4D these conditions %\footnote{We will explain how to obtain them from the viewpoint of differential invariants in Appendix \ref{ApA}.}
imply that the metric $g$ is $I$-degenerate \cite{CHP}. The opposite,
$I$-nondegeneracy of $g$, can be defined
through the map $I:\text{(spacetimes)}\to\text{(spi)}$ as
discreteness\footnote{In \cite{CHP} a weaker requirement is stated,
but the proof implies the  stated stronger property.}
of the set $I^{-1}(I(g))$ for germs of $g$ (in localization of $M$).
Note that for a generic metric, $I^{-1}(I(g))=g$ is a one point set.

In any dimension $n$ one can show
(by varying $H$ and $W_i$ in lower $v$-degree terms) that any degenerate Kundt metric $g$
can be smoothly deformed as a family $g_\tau$ with $I(g_\tau)=\op{const}$ and the deformation is not an isotopy.

 \begin{theorem}\label{T2}
The pseudogroup of local transformations preserving the
degenerate Kundt spacetimes of shape (\ref{Kundt})
coincides with the pseudogroup $\Sym$.
 \end{theorem}

 \begin{proof}
The conditions of degeneracy are natural (coordinate-independent)
and therefore are respected by any pseudogroup of transformations on $M$.
On the other hand, the class of degenerate Kundt metrics also
specifies the filtration $\l\subset\l^\perp$ used in the
preceding proof, and the shape is the same, whence the claim.
 \end{proof}

Adding the degeneracy condition to Kundt spacetimes determines a new PDE, denoted by $\ED$, which is specified by the equations
 $$
(h_{ij})_v =0,\quad (W_i)_{vv}=0,\quad H_{vvv}=0.
 $$
Including the prolongations of those conditions (that is applying total derivatives of all orders and directions)
we get the infinitely prolonged system $\ED^\infty\subset J^\infty \pi$.

More precisely, we have $\ED^k=\E^k$ for $k<2$, the submanifold $\ED^2\subset\E^2$ is given by the additional equations $(W_i)_{vv}=0$, and $\ED^3\subset\E^3$ by first derivatives of those
plus the equation $H_{vvv}=0$, etc.

By virtue of Theorem \ref{T2} the lift and prolongations of
the pseudogroup $\Sym$
restrict to the equation $\ED$ of degenerate Kundt metrics.
Differential invariants of order $k$ are rational-polynomial
functions $f$ on $\ED^k$ satisfying Lie equation \eqref{LieDI}.
Though the main target is the class
of degenerate Kundt spacetimes,
we can study simultaneously the class of general Kundt metrics.

\section{Counting the invariants}\label{S3}

Now we count the amount of (algebraically) % hence functionally
independent differential invariants for both general and degenerate
Kundt spacetimes, depending on the jet-order $k$.

\subsection{Jets and equations}

At first we determine dimensions of the involved jet-spaces and
equation-manifolds. With $N=\binom{n}2$ from \eqref{eN} we have
 \[
\dim J^k \pi = n+N \binom{n+k}{n}.
 \]

There are $\binom{n-1}{2}\binom{n+k-1}{n}$ equations of order $\leq k$
specifying Kundt spacetimes of the form (\ref{Kundt}).
This number is the codimension of $\E^k \subset J^k \pi$, whence
 \begin{align*}
\dim \E^k &= n+(n-1)\binom{n+k}{n}+\binom{n-1}{2}\binom{n+k-1}{n-1}\\
&= n+(n-1)\binom{n+k-1}{n}\frac{n^2+2k}{2k}
\text{ for }\ k>0
 \end{align*}
and $\dim\E^0=n+N=\binom{n+1}{2}$.

The equation-manifolds for degenerate Kundt metrics satisfy
$\ED^0=\E^0$ and $\ED^1=\E^1$, while
$\ED^k\subset\E^k$ is given by $(n-2)$ additional constraints for $k=2$
and by $(n-2)\binom{n+k-2}{n}+\binom{n+k-3}{n}$
constraints for $k\ge3$. Thus
 \begin{align*}
\dim\ED^2 &= \left(\binom{n+1}{2}+1\right)\left(\binom{n}{2}+1\right), \\
\dim\ED^k &= \dim\E^k- (n-2)\binom{n+k-2}{n}-\binom{n+k-3}{n}\
\text{ for }\ k>2.
 \end{align*}

\subsection{Orbit dimensions}\label{S32}

The action of $\Sym$ on $J^0\pi$ is transitive, so that any point can be mapped to the point $p_0$ given by
 \[
u=0,\ x^i=0,\ v=0,\ h_{ij}=\delta_{ij},\ W_i=0,\ H=0.
 \]
The stabilizer in $\sym$ of the point $p_0$ is given by
 \begin{equation*}
a^i=b=c=b_u=0,\ b_i=-2 a^i_u,\ a^i_j=-a^j_i.
 \end{equation*}
These conditions imposed on the jets of pseudogroup elements preserving
$p_0$ define an algebraic (finite-dimensional) group $\Sym^{(k)}_0$
acting in the fibers $J^k_0\pi\supset\E^k_0\supset\ED^k_0$ over $p_0$.
The invariants of this action bijectively correspond to differential
invariants of order $k$ of $\Sym$.

By Rosenlicht's theorem for an algebraic action its field of
rational invariants separates generic orbits, and the transcendence
degree of this field is equal to the codimension of a generic orbit.
Of course, codimension of orbits of $\Sym$ on $\E^k$ or $\ED^k$
equals codimension of orbits of $\Sym_0$ on $\E^k_0$ or $\ED^k_0$, respectively.

 \begin{theorem}\label{T3}
For $k=1$ the codimension of an orbit in general position in
$\E^1=\ED^1$ is 1.
For $k\geq2$ the dimension of an orbit in general position both in
$\E^k$ and in $\ED^k$ is given by
 \[
(n-1) \binom{n+k}{n-1}+k+2.
 \]
 \end{theorem}

 \begin{proof}
Consider the action of the stabilizer $\Sym^{(1)}_0$ on $\E^1_0$.
A straightforward verification shows that $\sum_{i=1}^{n-2} (W_i)_v^2$ is
an invariant. Now we use the pseudogroup to normalize a point in $\E^1_0$ by sequentially fixing a set of coordinates, thus restricting to a sequence of submanifolds. We simultaneously fix parameters of the Lie algebra, so that the remaining vector fields are tangent to the current submanifold.
 \begin{enumerate}
\item
Bring the point to the submanifold given by $(h_{ij})_k=0,\ (h_{ij})_u=0$.
The Lie subalgebra preserving the submanifold is restricted further by
$a^i_{jk}=0, a^i_{ju}=-a^j_{iu}$.
\item
Fix $(W_i)_j=0$. The stabilizer of this new submanifold is given by the additional equations $b_{ij}=(W_i)_va^j_u+(W_j)_va^i_u$,
$a^i_{ju}= \frac12((W_i)_va^j_u-(W_j)_va^i_u)$.
\item
Fix $(W_i)_u=0$. The new stabilizer is given by $b_{iu} =-2a^i_{uu}$.
\item
Fix $H_u=H_i=H_v=0$. The new stabilizer is given by $a^i_{uu}=b_{uu}=0$,
$c_{uu}=(W_i)_v a^i_u$.
\item
The remaining stabilizer is $CO(n-2)\ltimes\R^{n-2}$, and its subgroup $SO(n-2)$ acts nontrivially on the covector $(W_i)_v$, so we fix it so: 
$(W_2)_v=\dots=(W_{n-2})_v=0$. Then $(W_1)_v^2$ is the value of the above invariant.
\end{enumerate}
For the action of $\Sym^{(1)}_0$ on $\E^1_0$ the stabilizer of
a generic point $p_1$ has dimension $\binom{n-2}2+2$,
in particular the action is not free.

The same approach works in higher jets: by choosing a specific point
$p_k\in\ED^k_0$ we compute the rank of all $k$-jets of
vector fields $\xi\in\g$ at $p_k$. The totality of those fields may
be thought to be the number of free jets of group parameters
entering the fields $\xi^{(k)}$, which is
$(n-1)\binom{n+k}{n-1}+k+3$.
However, over $p_0$ the coefficient of $c^{(k+2)}$ vanishes (since $v=0$),
the corresponding field is in the kernel of the action, % non-effective
and therefore the group $\Sym_0^{(k)}$ has dimension 1 less than
the indicated number.

Now a tedious verification, which we omit, shows that these vector fields
are actually independent, so the orbit has the dimension as stated,
and the action is free for $k\ge2$ (by definition this means
that $\Sym_0^{(k)}$ acts freely).

An alternative route is to check (in the same manner as for 1-jets) that
the action is free on a Zariski open subset of $\ED^2_0$ (hence also
on a Zariski open subset of $\E^2_0$). Therefore, from the persistence of
freeness in prolongation \cite{OP}, the claim follows.
 \end{proof}

\subsection{Hilbert and Poincar\'e functions}\label{S33}

Let $s_k^n$ denote the codimension of an orbit in general position in $\E^k$. Define the Hilbert function for the action of $\Sym$ on $\E$
(or for the quotient $\E/\Sym$):
$\HK_k^n=s_k^n-s_{k-1}^n$ and $\HK_0^n=s_0^n$.
From Theorem \ref{T3} we conclude:
 \begin{proposition}
The Hilbert function for $\E/\Sym$ is given by $\HK_0^n=0$, $\HK_1^n=1$,
\begin{align*}
\HK_2^n &= n-5+(n-1) \left(\tbinom{n+2}{n}-\tbinom{n+2}{n-1}\right)+\tbinom{n-1}{2} \tbinom{n+1}{n-1} \\
&= \frac{n^4-4n^3+11n^2+16n-72}{12}, \\
\HK_k^n &= (n-1)\left( \tbinom{n+k-1}{n-1}-\tbinom{n+k-1}{n-2}\right)+\tbinom{n-1}{2} \tbinom{n+k-2}{n-2} -1\ \text{ for }k\geq3.
\end{align*}
 \end{proposition}

 \begin{corollary}
For $k\geq3$ the Hilbert function in dimensions $n=3,4,5$ is given by
 \begin{align*}
\HK_k^3 &= k^2+2k-2, \\
\HK_k^4 &=\frac{1}{2} \left( k^3+6k^2+5k-8\right), \\
\HK_k^5 &= \frac{1}{6}\left( k^4+12k^3+35k^2+12k-42\right).
 \end{align*}
For $k=2$ we have $\HK_2^3=4$, $\HK_2^4=14$ and $\HK_2^5= 34$.
 \end{corollary}

Similarly define $\tilde{s}_k^n$ and $\HD_k^n=\tilde{s}_k^n-\tilde{s}_{k-1}^n$, $\HD_0^n=\tilde{s}_0^n$ 
for the action of $\Sym$ on $\ED^k$. In the same manner
as Theorem \ref{T3} we conclude:

 \begin{proposition}
The Hilbert function for $\ED/\Sym$ is given by $\HD_0^n=0$, $\HD_1^n=1$,
 \begin{align*}
\HD_2^n &= (n-1) \left( \binom{n+2}{2}-\binom{n+2}{3} \right) +\binom{n-1}{2} \binom{n+1}{2}-3,\\
\HD_k^n &= (n-1) \left( \binom{n+k-1}{n-1} -\binom{n+k-1}{n-2} \right) + \binom{n-1}{2} \binom{n+k-2}{n-2} \\
&-(n-2) \binom{n+k-3}{n-1} -\binom{n+k-4}{n-1} -1
\ \text{ for }k\geq3.
\end{align*}
 \end{proposition}

 \begin{corollary}
For $k\geq3$ the Hilbert function in dimensions $n=3,4,5$ is given by
 \begin{align*}
\HD_k^3 &= 4k-3, \\
\HD_k^4 &= \frac{1}{2} (7k^2+5k-8), \\
\HD_k^5 &= \frac{1}{6} (11 k^3+36 k^2+13 k-42).
 \end{align*}
For $k=2$ we have $\HD_2^3=3$, $\HD_2^4=12$ and $\HD_2^5=31$.
 \end{corollary}

Another way to encode the counting of invariants is through the Poincar\'e function
 $$
P_n(z)=\sum_{k=0}^\infty \HK_k^n z^k.
 $$
Since the Hilbert function is polynomial in $k\ge k_0$, the Poincar\'e function is rational.
 \begin{corollary}
The Poincar\'e function in dimensions $n=3,4,5$ is given by
 \begin{align*}
P_3(z) &= \frac{(1+z+4 z^2-6 z^3+2 z^4) z}{(1-z)^3},\\
P_4(z) &= \frac{(1+10 z-6 z^2-10 z^3+11 z^4-3 z^5) z}{(1-z)^4}, \\
P_5(z) &= \frac{(1+29 z-41 z^2+33 z^4-23 z^5+ 5 z^6) z}{(1-z)^5}
 \end{align*}
for general Kundt spacetimes; for degenerate Kundt spacetimes it is
 \begin{align*}
\tilde{P}_3(z) &= \frac{(1+z+4 z^2-2 z^3) z}{(1-z)^2},\\
\tilde{P}_4(z) &= \frac{(1+9 z+2 z^2-8 z^3+3 z^4) z}{(1-z)^3}, \\
\tilde{P}_5(z) &= \frac{(1+27 z-15 z^2-15 z^3+18 z^4-5 z^5) z}{(1-z)^4}.
 \end{align*}
 \end{corollary}

\section{Computing the invariants}\label{S4} % of Kundt spacetimes

There are several approaches for  describing the algebra of
invariants by generators and syzygies in Lie-Tresse type framework
discussed in Section \ref{S21}. We first give a common scheme,
and then specify it for general and degenerate Kundt spacetimes.
Afterwards we provide an alternative approach with simpler computations
in low dimensions $n=3,4$.

\subsection{The general scheme} \label{S41}

One general approach is to find $n$ \emph{horizontally independent}\footnote{Horizontal independence implies algebraic independence (in jets), i.e. $\op{rank}(\p_{J^{i,j}_\sigma}I_s)=n$, where $J^{i,j}_\sigma$ consists of the base variables $x^i$ and the
jet-variables $y^j_\sigma$. But $n$ invariants can be algebraically independent without being horizontally independent.} rational differential
invariants $I_1,\dots,I_n$, i.e. invariants satisfying
\begin{equation} \label{Independent}
\det[\D_iI_s]\not\equiv0.
\end{equation}
The condition \eqref{Independent} means that for a generic section $\psi\in\Gamma(\pi)$, 
the restriction $\bar{I}_s=(j^\infty\psi)^*I_s$ of the above invariants 
to the holonomic jet-section $j^\infty\psi$ are functionally 
independent; since under this restriction they become functions on $M$
this can be writen as follows: 
 $$
\det[\p_i\bar{I}_s]\not\equiv0.
 $$

Next, derive the corresponding horizontal\footnote{The horizontal
differential is defined by the formula $\hat{d}f|_{j^{k+1} \psi}=d(f\circ j^k \psi)$
$\forall f\in C^\infty(J^k\pi),\ \psi \in\Gamma(\pi)$.} coframe
$\omega^i=\hat{d}I_i$ and its dual horizontal frame $\nabla_i=\D_{I_i}$. This particular set of invariant derivations $\nabla_i$ are called Tresse derivatives, and they are pairwise commuting. When restricted to a generic section of $\pi$, they reduce to partial derivatives with respect to $\bar{I}_i$.
We express $g$ in this frame:
 \begin{equation}\label{gI}
g=G_{ij}\omega^i\omega^j, \qquad G_{ij} = g(\nabla_i,\nabla_j)
 \end{equation}
Now the algebra $\mathcal{A}$ of differential invariants is generated
by $I_i, G_{ij}$ and $\nabla_i$.  %, equivalently we call $(I_k,G_{ij})$ a fundamental system of invariants.

Indeed, in any coordinate system $(x^i)$ the invariants are obtained
from the invariant combinations of the components $g_{ij}$ of the metric, and
their partial derivatives. If we choose invariant
coordinates $\bar{I}_i$, then the metric components and their derivatives
are also invariants. Since no invariants are lost during the change of coordinates, all invariants are obtained as derivatives of the components with respect to $\bar{I}_i$.

Note that the passage $(x^1,\dots,x^n)\mapsto(I_1,\dots,I_n)$
is a differential operator, without differential inverse in general.
Therefore the count of invariants in Section \ref{S3} does not
survive this transformation. %In particular, the Hilbert functions will differ. 
However, the asymptotics of the Hilbert function\footnote{If
$P(z)=\frac{R(z)}{(1-z)^d}$ is the Poincar\'e function,
with a polynomial $R(z)$ not divisible by $(1-z)$, then the asymptotic
is encoded by the numbers $d$ and $\sigma=R(1)$, see \cite{Kr}.} do
survive. For general metrics the asymptotics are given by $d=n$, $\sigma=\binom{n}2$.
For general Kundt spacetimes $d=n$, $\sigma=n-1$.
For degenerate Kundt spacetimes $d=n-1$, $\sigma=\binom{n}2+1$.
This tells us that, modulo diffeomorphism (coordinate) freedom,
the metrics in the class locally depend on $\sigma$ arbitrary
functions of $d$ variables.

The requirement \eqref{Independent} allows for a wide variety of possibilities when it comes to choosing the $n$ scalar differential invariants $I_1, \cdots, I_n$. For example, for generic Kundt metrics, they can be taken as normalized components\footnote{Beware that normalization can result in elements of an algebraic extension of the field of rational invariants. In particular, invariants obtained in this way may contain roots.}
of the Riemann tensor (cci: Cartan invariants), i.e.\ through
its Ricci or Weyl components, cf.\ \cite{ESEFE}. They can
 also be taken as spi. For instance, following \cite{LY}, choose
 \[
I_1=\op{Tr}(\op{Ric}_g),\ \dots,\ I_n=\op{Tr}(\op{Ric}_g^n).
 \]
There are other possibilities, as we will show in detail for dimensions 3 and 4.  

The differential invariants $G_{ij}$ are rational functions (also when $I_i$ are psi), and so are their Tresse derivatives. Notice however that we have good control of the domain where these rational invariants are defined.  The condition \eqref{Independent} is equivalent to 
\[\hat dI_1 \wedge \cdots  \wedge \hat dI_n \neq 0.\] Assume that the $n$ horizontally independent invariants are of order $k$ or less, and let $\Sigma \subset \E^{k+1}$ (or  $\Sigma \subset \ED^{k+1}$ in the case of degenerate Kundt) denote the set on which $\hat dI_1 \wedge \cdots \wedge \hat dI_n$ vanishes or diverges.  Then $G_{ij}$ are defined on $\E^{k+1}\setminus \Sigma$. Moreover, the derivatives $\nabla_{i_1} \circ \cdots \circ \nabla_{i_r}(G_{ij})$ are defined on $\pi_{k+r+1,k+1}^{-1}(\Sigma) \cap \E^{k+r+1}$, and their restrictions to fibers of $\E^{k+r+1} \to \E^{k+1}$ are polynomials. We refer to \cite{KL2} for more details.  

Another way of finding a generating set of invariants is to construct $n$ independent invariant derivations
$\nabla_1,\dots,\nabla_n$ that are not Tresse derivatives. These form a horizontal frame with dual
horizontal coframe $\omega^1,\dots,\omega^n$, which in turn
determines differential invariants $G_{ij}$ via \eqref{gI}. This lets us again generate $\mathcal{A}$ if we include, in the set of generators, the structure functions $c_{ij}^k$ from $[\nabla_i,\nabla_j] = c_{ij}^k \nabla_k$. The analysis of singular points in the previous paragraph can be adapted to this setting. 

%If there is a horizontally independent set $J_1,\dots, J_n$ among $G_{ij}$, the invariants $c_{ij}^k$ comes automatically when we compute derivatives of $J_i$, since in that case the matrix $\nabla_k(J_l)$ with invariant entries is nondegenerate. In the case of degenerate Kundt spacetimes, we will encounter situations where the set $G_{ij}$ contain only $n-1$ horizontally independent invariants, in which case $G_{ij}$ and $\nabla_i$ will not automatically generate the whole algebra of invariants.

\subsection{Invariants of Kundt spacetimes} \label{S42}

For general Kundt metrics we can, as discussed above, use the Ricci operator $\op{Ric}_g$ and $n$ second-order differential invariants $I_i=\op{Tr}(\op{Ric}_g^i)$, $1\leq i\leq n$. (The restriction of the horizontal tensor field $\op{Ric}_g$ to a Kundt spacetime is an operator 
$TM\to TM$.)
The invariants $I_1,\dots, I_n$ are horizontally independent on a Zariski open set of 3-jets of Kundt metrics,
and thus are sufficient to generate the entire algebra $\mathcal{A}$ of sdi as explained above.

However, for degenerate Kundt spacetimes there are less than $n$
horizontally independent functions among $I_i$. Actually, the Ricci operator
in the $(u,x,v)$ coordinates adapted to Kundt alignment has the form
 $$
\op{Ric}_g=\begin{bmatrix}
\lambda & 0 & 0 \\ * & R_h & 0 \\ * & * & \lambda
\end{bmatrix}
 $$
with $R_h$ being determined by the Ricci operator for the Riemannian
metric $h_{ij}$ on $\bar{M}_u$ and the 2-jet of $W_i$ (more precisely by $(W_i)_v$ and $(W_i)_{x^j v}$) in an invariant manner. When restricted to a degenerate Kundt spacetime, the block-diagonal entries of the operator depend only
on $(u,x)$. Thus, the eigenvalues are $v$-independent functions, and the maximal number of functionally independent eigenvalues is $(n-1)$. For generic degenerate Kundt metrics, this upper bound is reached and the rank of the total Jacobian matrix $[\D_i I_j]$ is equal to $n-1$.

Let $I_1,\dots,I_{n-1}$ be horizontally independent invariants 
chosen from the above set. For degenerate Kundt spacetimes we have $\p_v\bar{I}_i=0$. 
The annihilator of restricted invariants
$d(I_i|_{j^\infty{\psi}})=\hat{d}I_i|_{j^\infty{\psi}}$ integrates to
the foliation $\lambda$ of dimension 1;
here $I_i|_{j^\infty{\psi}}$ is the pullback by the jet-section
$j^\infty{\psi}$ of $J^\infty\pi$ (equivalently: evaluated on the Kundt metric defined by $\psi$),
that we also denoted $\bar{I}_i$, and similar for 1-forms. 

%For generic Kundt spacetimes there exists an invariant with 
%$\p_v\bar{I}_i\neq0$; by renumeration we can assume $i=n$. 
%The first invariant derivation is then defined by the condition
% $$
%\nabla_1\|\l,\ \hat{d}I_n(\nabla_1)=1\ \Longleftrightarrow \ \nabla_1=\frac1{\D_v(I_n)}\D_v
% $$
%In this case the rank $(n-2)$ foliation
%$\{I_n=\op{const}\}\cap\Lambda$ is spacelike.

Consider the horizontal covectors $\hat{d}I_1,\dots,\hat{d}I_{n-1}$ and 
$g$-dual horizontal vector fields $\nabla_1,\dots,\nabla_{n-1}$
tangent to $\Lambda$ (we remind that $\Lambda$ is the foliation of codimension 1
with fibers tangent to $\lambda^\perp$). Since the restriction of $g$ to $\Lambda$
is non-negative definite with one-dimensional kernel, we can
without restriction of generality assume that 
the vectors $\nabla_2,\dots,\nabla_{n-1}$ determine a spacelike
subbundle of $\pi_\infty^*TM$ on a Zariski open set in jets.
We claim that the $(n-2)\times(n-2)$ Gram matrix is non-degenerate (and hence positive definite):
 $$
\det[g(\nabla_i,\nabla_j)]_{i,j=2}^{n-1}\not\equiv0.
 $$
%(Note that in general, upon evaluation on a metric $g$, 
%the distribution spanned by $\{\nabla_i\}_{i=2}^{n-1}$ is non-integrable,
%though if it was the spacelike leaves would be diffeomorphic to $\bar{M}$.)

Finally, we uniquely determine the last invariant derivation $\nabla_n$
by the conditions
 $$
g(\nabla_1,\nabla_n)=1,\ g(\nabla_i,\nabla_n)=0 \text{ for }1<i\leq n.
 $$
In fact, since restriction of $g$ to the rank two distribution 
$\langle\nabla_2,\dots,\nabla_{n-1}\rangle^\perp$ is Lorentzian,
it has precisely two null-directions at each point. 
One is $\D_v$, and $\nabla_1$ is projected to it along 
$\langle\nabla_2,\dots,\nabla_{n-1}\rangle$ (its projection is also an invariant derivation). The other null-direction
is spanned by $\nabla_n$.

This gives an invariant frame, i.e.\ a basis of sections
of the Cartan distribution $\Cc\simeq\pi_\infty^*TM$,
and the algebra $\mathcal{A}$ is determined by this general scheme.

\subsection{The algebra of differential invariants in low dimensions}

The algorithm considered above provides a complete set of differential
invariants, but the generators have high algebraic complexity.
Therefore, in what follows, we provide an alternative simpler description
of the algebra $\mathcal{A}$ in important dimensions $n=3,4$.

We begin with a general remark. As we saw in Section \ref{S3}, the
action of $\Sym$ is transitive on $J^0 \pi$ and has precisely 1 differential
invariant of order 1 for both $\E$ and $\ED$ in any dimension $n$. We will recycle the notation $I_i$ and $\nabla_i$ from Section \ref{S41} and Section \ref{S42}, and we will continue doing so in Section \ref{S44} to Section \ref{S47}. 

 \begin{proposition}\label{propI1}
Let $w=(W_i)_v dx^i$. The first-order differential invariant is given by 
 \[
I_1= \|{w}\|^2_g= (W_i)_v (W_j)_v h^{ij}.
 \]
 \end{proposition}
Here $[h^{ij}]$ is the inverse of the symmetric matrix consisting of fiber coordinates $h_{ij}$.

Since the foliation $\lambda$ is internally invariant, it is reasonable to look for a derivation of the form  
$$
\nabla_1=\gamma \D_v.
$$
For general Kundt spacetimes, we have $\D_v(I_1) \not\equiv 0$, which means that the factor $\gamma$ can be determined by the condition $\nabla_1(I_1)=2$.\footnote{The constant $2$ is a convenient choice with our coordinates. In principle, the right-hand-side can be set equal to any differential invariant.} The invariant $I_1$ determines the invariant derivation $\nabla_2= g^{-1}\hat{d}I_1$ which has, for general Kundt spacetimes, a nonzero $\D_u$-component. 

If $n=3$, we can complete the frame with a derivation $\nabla_3$ which is determined (up to an overall sign) by the equations 
$$
g(\nabla_1,\nabla_3)=0, \qquad g(\nabla_2,\nabla_3)=0, \qquad g(\nabla_3,\nabla_3)=4/I_1.
$$

If $n=4$, we find the third derivation $\nabla_3$ in a different way. Let $\nabla_3= \mathbb{D}^{g}_{\nabla_2} \nabla_1$, where $\mathbb{D}^g$ denotes the covariant derivative with respect to the Levi-Civita connection.%\footnote{We could also have defined $\nabla_3=[\nabla_1,\nabla_2]$. However, the current choice results in simple $\D_{x^i}$-components, and vanishing $\D_u$-component.} 
The invariant horizontal frame can now be completed (up to an overall sign) by $\nabla_4$ satisfying 
$$
g(\nabla_1,\nabla_4)=0, \qquad g(\nabla_2,\nabla_4)=0,\qquad g(\nabla_3,\nabla_4)=0, \qquad g(\nabla_4,\nabla_4)=I_1.
$$
%Note that $\nabla_4$ depends on square roots of the determinant of $g$, but is otherwise rational. 
%Alternatively, use the derivation $[\nabla_1, \nabla_3]$ instead of $\nabla_4$. 

In the case of degenerate Kundt spacetimes, the above approach can still be used, but since $\D_v(I_1) \equiv 0$ for degenerate Kundt spacetimes, $I_1$ must be replaced with a different invariant $I$ which satisfies $\D_v(I)\not\equiv 0$. 

In what follows we write down explicitly a basis of invariant derivations for $n=3$ and $n=4$, slightly different than the ones suggested above. By the discussion in Section \ref{S41}, such an invariant horizontal frame solves the equivalence problem. For $n=3$ we also write down explicitly a transcendence basis for the field $\mathcal{A}_2$ of second-order differential invariants. 

We have used the computer algebra system Maple, with the \texttt{DifferentialGeometry} and \texttt{JetCalculus} packages, when computing with invariant derivations and differential invariants. These packages provide an easy way to verify the statements in this section that rely on symbolic computations. 

\subsection{General 3D Kundt spacetimes} \label{S44}

In this and the next subsection (when we consider $n=3$) we simplify the notation: 
$W_1=W, h_{11}=h, x^1=x$.
Then the invariant of Proposition \ref{propI1} is given by
 \begin{equation}\label{I13D}
I_1=\frac{W_v^2}{h}.
 \end{equation}
 Note that the function $W_v/\sqrt{h}$ is invariant with respect to the connected component of $\Sym$ in the smooth topology. However, it changes sign under the transformation 
 \[(u,x,v,h,W,H)\mapsto (u,-x,v,h,-W,H)\] 
 which is contained in its Zariski closure. 

%We find three independent invariant derivations in the following way. We use the first-order invariant to construct the invariant derivation \[\nabla_1= g^{-1}\hat dI_1.\] (There may be a better choice of $\nabla_1$. The important thing is that the $D_u$-component is nonzero. In the case of degenerate Kundt spacetimes we need to find another $\nabla_1$.)
%
%The direction $\partial_v$ is singled out for Kundt spacetimes. There is an invariant derivation of the form $\gamma D_v$. We determine $\gamma$ by imposing an invariant condition, such as $g(\nabla_1,\gamma D_v)=2$.

%An invariant derivation is a derivation $\nabla=\alpha D_x+\beta D_u+\gamma D_v$ satisfying $[\nabla,X^{(\infty)}]=0$ for every $X \in \mathfrak g$. Here $\alpha,\beta,\gamma$ are functions on $\E^k$ for some $k$. The following statement is easily verified by direct computation.
The following proposition is easily verified.
 \begin{proposition}
The derivations
\begin{align*}
\nabla_1 &= \frac{W_v}{W_{vv}} \D_v,\qquad 
\nabla_2 = \frac{2}{W_v} \D_x+\frac{h_x W_v-2 h W_{xv}}{h W_v W_{vv}} \D_v, \\
\nabla_3 &=\frac{1}{W_v} \left(H_{vv} \D_x-W_{vv} \D_u+(W_{uv}-H_{xv}) \D_v \right)
\end{align*}
are invariant, and they are independent on a Zariski open subset of $\E^2$. 
 \end{proposition}
%The commutation relations are given by $[\nabla_i,\nabla_j]=K_{ij}^k \nabla_k$ where $K_{ij}^k$ are differential invariants of order three. They are computed in the Maple file. In particular 
We have $[\nabla_1,\nabla_2]=-\nabla_2$. The other commutation relations contain nontrivial structure functions (and new invariants), but we omit their explicit form due to their length.

Let $\alpha^j$ denote the elements of the dual horizontal coframe (defined by $\langle \nabla_i,\alpha^j\rangle=\delta_i^j$). The horizontal symmetric 2-form $g$ written in terms of this coframe will have coefficients given by $g(\nabla_i,\nabla_j)$. It takes the form
 \[ 
g = I_1^{-1} \left((J_1 \alpha^3+J_2 \alpha^2-I_1 \alpha^1)\alpha^3  + 4 (\alpha^2)^2\right)
 \]
where
 \begin{align*}
J_1&=\frac{H W_{vv}^2+(- H_{vv} W+ H_{xv}- W_{uv}) W_{vv}+H_{vv}^2 h}{h},\\
J_2&= \frac{4H_{vv} h^2+2(W_{xv}-W W_{vv}) h-W_v h_x}{h^2}.
 \end{align*}

Let us recall from Section \ref{S3} that there are 4 algebraically independent second-order invariants (excluding the one of first order). 

\begin{proposition}\label{th:3Dsecondorder}
The five differential invariants $I_1, J_1,J_2$ and
\begin{align*}
\nabla_3(I_1) &= 2\,\frac{H_{vv} W_{xv}-H_{xv} W_{vv}}{h}-\frac{W_v (H_{vv} h_x-W_{vv} h_u)}{h^2},\\
J_3 &= \frac{W_{vv}^2(h_u^2-2hh_{uu})}{h^3}-\frac{2 W_{vv} (H_v W_{vv}-H_{vv} W_v) h_u}{h^2}\\&-\frac{((H_v W-H_x+W_u) W_{vv}^2-W_v (H_{vv} W-H_{xv}+W_{uv}) W_{vv}+2 H_{vv}^2 h W_v) h_x}{h^3}\\ & +\frac{(-2 H_x W_v+2 H_v W_x+2 H_{xv} W-2 H_{xx}+2 W_{ux}) W_{vv}^2}{h^2} \\
&+\frac{((-2 H_{vv} W+2 H_{xv}-2 W_{uv}) W_{xv}-4 H_{xv} H_{vv} h) W_{vv}+4 H_{vv}^2 h W_{xv}}{h^2}
%J_3 &= \frac{1}{h^4} \Big(
%2 W_{vv}^2 h_{uu} W_{v}^2 h+H_{vv}^2 h_{x}^2 W_{v}^2-2 H_{vv} W_{vv} h_{x} h_{u} W_{v}^2\\
%&+((-H W_{vv}^2+H_{vv}^2 h) W_{v}^3 +(H_{v} W-H_{x}+W_{u}) W_{vv}^2 W_{v}^2 \\
%&+4 H_{vv} h (H_{xv} W_{vv}-H_{vv} W_{xv}) W_{v}) h_{x}+(-2 H_{vv} W_{vv} h W_{v}^3 \\
%&+2 H_{v} h W_{vv}^2 W_{v}^2-4 h W_{vv} (H_{xv} W_{vv}-H_{vv} W_{xv}) W_{v}) h_{u} \\
%&+2 H_{x} W_{v}^3 W_{vv}^2 h+(-2 H W_{vv}^3 W h+(4 H H_{vv} h^2 \\
%&+(2 H_{vv} W^2+2 H W_{xv}-2 H_{v} W_{x}-4 H_{xv} W+2 W W_{uv}\\
%&+2 H_{xx}-2 W_{xx}) h) W_{vv}^2+2 H_{vv} (-3 H_{vv} W+4H_{xv}-2 W_{uv}) h^2 W_{vv} \\
%&+4 H_{vv}^3 h^3-2 H_{vv}^2 h^2 W_{xv}) W_{v}^2+4 h^2 (H_{xv} W_{vv}-H_{vv} W_{xv})^2 \Big).
\end{align*}
constitute a transcendence basis for the field of second-order differential invariants on $\E^2$.
\end{proposition}

Note that $\nabla_2(I_1)=0$ and $\nabla_3(I_1)=2I_1$. By differentiating $J_1,J_2,J_3,\nabla_1(I_1)$ with respect to $\nabla_1,\nabla_2,\nabla_3$, we get 12 differential invariants of order 3, while $\HK_3^3=13$. The differential invariant
\[K_{13}^1-1= -\frac{W_{vvv} W_v}{W_{vv}^2}\] is algebraically independent from the others, and thus completes the transcendence basis for the field of third-order invariants. 

By adding to the 12 third-order invariants the 9 second-order derivatives of $J_1$,$J_2$, $J_3$, $\nabla_1(I_1)$ (36 in total), we get 40 algebraically independent differential invariants of order 4, which generate a transcendence basis for the field of fourth-order differential invariants.

Since $\hat dI_1 \wedge \hat dJ_1 \wedge \hat d J_2 \not\equiv 0$, we can obtain all the structure functions from the commutation relations by differentiating these three invariants. Since the coefficients of the metric in the chosen frame are effectively $I_1, J_1, J_2$, we obtain the following statement. 

\begin{theorem}
For $n=3$ the algebra $\mathcal{A}$ of differential invariants 
of the $\Sym$ action on $\E$ is generated by the differential invariants
$I_1,J_1,J_2$ and the invariant derivations $\nabla_1,\nabla_2,\nabla_3$.
\end{theorem}

\subsection{Degenerate 3D Kundt spacetimes}
The function $I_1$ of the form \eqref{I13D} is a differential invariant also in the case of degenerate Kundt spacetimes, since $\ED^1=\E^1$ and the Lie pseudogroup action is the same. From Section \ref{S33} we know that there are, in addition, 3 algebraically independent differential invariants of order 2. It is possible to restrict the invariants on $\E^2$ to $\ED^2$, but our transcendence basis on $\E^2$ does not restrict to a transcendence basis on $\ED^2$. %Indeed, in the degenerate case $D_v(I_1)=0$. 

%The 5 second-order differential invariants for the general Kundt $\E_2$
%can be restricted to $\ED_2$. However this yields only 3 independent invariants. . Thus the 
%invariants do not separate orbits. We need to add an additional invariant $I_{2c}$ to complete the description of $\mathcal{A}_2$.

Let us first define the following:
\[
 I_{2a} = H_{vv}, \quad I_{2b} = \frac{W_v h_x-2h W_{xv}}{h^2}, \quad K_{2a} = \frac{H_{xv}-W_{uv}}{W}, \quad K_{2b}=\frac{W_v h_u-2h W_{uv}}{W h}.
\]
The functions $I_{2a}$ and $I_{2b}$ are second-order differential invariants on $\ED^2$. The functions $K_{2a}$ and $K_{2b}$ are not invariant, but will be convenient for simplifying the formulas in this subsection. For the same reason, we also introduce the (non-invariant) functions
\begin{align*}
Q &= \frac{(2 I_{2a} K_{2b}+I_{2b} K_{2a}-I_{2a} I_{2b}) W}{I_1}, \\
R &= \frac{I_{2b} H W_v^2}{I_1}-\frac{(I_{2b} I_{2a}^2-2 K_{2a} (I_{2b}-2 K_{2b}) I_{2a}+I_{2b} K_{2a}^2) W^2}{4 I_{2a}^2}.
\end{align*} 
A fourth second-order differential invariant is given by 
\begin{align*}
I_{2c}=&\frac{1}{Q^2} \Big(\frac{(I_1^2 h_{u} ( W W_v+h_{u})-(h_{x} (H_v  W-H_x+W_u) I_1^2-W_v^4 I_{2b} H)) I_{2a} I_{2b}}{W_v^2}\\
-&2  \left(W_v H_x+(h_{u}-W_x) H_v+H_{xx}-W_{ux}+h_{uu}\right) I_1 I_{2a} I_{2b} \\
- &W^2 I_1 (K_{2b} (I_{2b}-K_{2b}) I_{2a}-2 I_{2b} K_{2a}^2)
\Big).
\end{align*}

 \begin{proposition}
The differential invariants  $I_1, I_{2a}, I_{2b}, I_{2c}$ constitute a transcendence basis for the field of second-order differential invariants on $\ED^2$. 
 \end{proposition}
% \begin{proof}
% Verifying that these are independent and invariant is a straight-forward algebraic computation. Since the codimension of an orbit in $\ED_2$ in general position is $4$, the proposition follows.
% \end{proof}
 
 Notice that $\D_v(I_1)=\D_v(I_{2a})=\D_v(I_{2b})=0$ on $\ED^3$. Therefore, $\hat dI_1 \wedge \hat dI_{2a} \wedge \hat d I_{2b}=0$ everywhere. On the other hand, we have $\hat dI_1 \wedge \hat dI_{2a} \wedge \hat d I_{2c}\neq 0$ on a Zariski open set in $\ED^3$. Since $I_1, I_{2a}, I_{2c}$ are horizontally independent, we can write $g$ in terms of them, as explained in Section \ref{S41}, and in this way generate the whole algebra of differential invariants. 
 
Alternatively, we can express the metric in terms of an invariant horizontal frame. 
\begin{proposition}
The derivations 
\begin{align*}
\nabla_1 &= \frac{I_1}{I_{2a} I_{2b}} \cdot \frac{Q}{W_v} \D_v, \qquad \nabla_2 = \frac{1}{W_v} \left( \D_x-\frac{K_{2a}}{I_{2a}} W \D_v \right), \\
\nabla_3 &= \frac{2 I_{2a}}{I_1} \cdot \frac{1}{Q W_v} \left( K_{2b}  W \D_x-I_{2b} h \D_u+ R \D_v \right) 
\end{align*}
are invariant, and they are independent on a Zariski open subset of $\ED^2$. 
\end{proposition}
Notice that $\nabla_1$ and $\nabla_2$ can be simplified by multiplying by invariant functions. We have kept these factors because the metric has simple coefficients when expressed in terms of this horizontal frame. If we denote by $\alpha^1, \alpha^2,\alpha^3$ the horizontal coframe dual to the horizontal frame $\nabla_1, \nabla_2, \nabla_3$, we have
\[g=I_1^{-1} \left( (-2 \alpha^1+2 \alpha^2+\alpha^3) \alpha^3+(\alpha^2)^2\right) .\]
It follows that the algebra of differential invariants is generated by $I_1, \nabla_1, \nabla_2, \nabla_3$ and the structure functions in the commutation relations. Since $\hat dI_1 \wedge \hat dI_{2a} \wedge \hat d I_{2c} \not\equiv 0$, the structure functions can be recovered by applying $\nabla_1,\nabla_2,\nabla_3$ to $I_1, I_{2a}, I_{2c}$.

\begin{theorem}
For $n=3$ the algebra $\mathcal{A}$ of differential invariants 
of the $\Sym$ action on $\ED$ is generated by the differential invariants
$I_1,I_{2a}, I_{2c}$ and the invariant derivations $\nabla_1,\nabla_2,\nabla_3$.
\end{theorem}

%By differentiating $I_{2a}, I_{2b}, I_{2c}$ we get 9 new independent invariants of order three, which is the desired amount. Note also that $I_{2b}= \nabla_2(I_1)$.

\begin{remark}
Here, we have written a frame of invariant derivations with coefficients in $\ED^2$. Allowing coefficients in $\ED^k$ for higher $k$, may allow for invariant derivations in more compact form, such as
\[\frac{(W H_{vv}-H_{xv}+W_{uv})H_{xvv}-2hH_{vv} H_{uvv}}{h} \D_v.\]
\end{remark}

\subsection{General 4D Kundt spacetimes}
For general four-dimensional Kundt spacetimes, the invariant of Proposition \ref{propI1} is given by
 \begin{equation}\label{I14D}
I_1=\frac{(W_1)_v^2 h_{22} -2 (W_1)_v (W_2)_v h_{12} + (W_2)_v^2 h_{11}}{h_{11} h_{22}-h_{12}^2}.
 \end{equation}
Let us introduce the notation 
\[ A=(W_1)_v h_{22}-(W_2)_v h_{12}, \qquad B= (W_1)_v h_{12}-(W_2)_v h_{11}, \]
and 
%\begin{align*}
%T= &\frac{1}{2 (h_{11} h_{22}-h_{12}^2) (A (W_1)_{vv}-B (W_2)_{vv})} \cdot \Big(\\ 
%&A^3(h_{11})_{x^1}-((h_{11})_{x^2}+2(h_{12})_{x^1})A^2B+((h_{22})_{x^1}+2(h_{12})_{x^2})AB^2-B^3(h_{22})_{x^2} \\
%- &2(h_{11}h_{22}-h_{12}^2)(A^2(W_1)_{x^1 v}+B^2(W_2)_{x^2 v}-AB((W_1)_{x^2 v}+(W_2)_{x^1 v}))\Big) .
%\end{align*}
\begin{multline*}
T =
\frac { \splitfrac{A^3(h_{11})_{x^1}-((h_{11})_{x^2}+2(h_{12})_{x^1})A^2B+((h_{22})_{x^1}+2(h_{12})_{x^2})AB^2-B^3(h_{22})_{x^2}}{
-2(h_{11}h_{22}-h_{12}^2)(A^2(W_1)_{x^1 v}+B^2(W_2)_{x^2 v}-AB((W_1)_{x^2 v}+(W_2)_{x^1 v}))
}}
{2 (h_{11} h_{22}-h_{12}^2) (A (W_1)_{vv}-B (W_2)_{vv})}.
\end{multline*}

We have the following proposition. 
\begin{proposition}
The derivations
\begin{gather*}
\nabla_1= \frac{(W_1)_v^2 h_{22} -2 (W_1)_v (W_2)_v h_{12} + (W_2)_v^2 h_{11}}{A (W_1)_{vv}-B (W_2)_{vv}} \D_v,  \\
\nabla_2 = \frac{(W_2)_{vv} \D_{x^1}-(W_1)_{vv} \D_{x^2}+((W_1)_{x^2v}-(W_2)_{x^1v}) \D_v}{(W_1)_v (W_2)_{vv}-(W_2)_v (W_1)_{vv}}, \\
\nabla_3 = \frac{A \D_{x^1} - B \D_{x^2}  + T \D_v}{h_{11} h_{22}-h_{12}^2} , \qquad \nabla_4= g^{-1} \hat dI_1 
\end{gather*}
are invariant, and they are independent on a Zariski open subset of $\E^2$.
\end{proposition}
Notice that $\nabla_4$ is the only derivation among these that have a non-zero $\D_u$-component. 
We have $g (\nabla_1,\nabla_i)=0$ for $i=1,2,3$, and 
\[ g(\nabla_2,\nabla_2)=I_{2a}, \; g(\nabla_2,\nabla_3)=1, \; g(\nabla_3,\nabla_3) = I_1, \; g(\nabla_1,\nabla_4)=2I_1, \; g(\nabla_3,\nabla_4)=0. \] 
Here 
\[I_{2a}=\frac{(W_1)_{vv}^2 h_{22}-2 (W_1)_{vv} (W_2)_{vv} h_{12} + (W_2)_{vv}^2 h_{11}}{((W_1)_v (W_2)_{vv}-(W_2)_v (W_1)_{vv})^2}\] is one of the second-order differential invariants. The formulas for $g(\nabla_2, \nabla_4)$ and $g(\nabla_4,\nabla_4)$ are more complicated. 

There are $\HK_2^4=14$ algebraically independent differential invariants of order 2, so we will not attempt to write down all of them. Instead we will be satisfied with finding four horizontally independent differential invariants. The scalar curvature $S_h$ of the ($u$-parametrized) metric $h$ is an invariant function depending only on $h_{ij}$ and their $x^i$-derivatives up to second order. A fourth differential invariant is given by 
\[I_{2b}= \frac{\big(((W_2)_{uv}-H_{x^2v}) (W_1)_{vv}-((W_1)_{uv}-H_{x^1v}) (W_2)_{vv}+((W_1)_{x^2v}-(W_2)_{x^1v}) H_{vv}\big)^2 }{h_{11} h_{22}-h_{12}^2}.\]

\begin{theorem}
The four differential invariants $I_1, I_{2a},I_{2b},S_h$ are horizontally independent on a Zariski open subset in $\E^3$, and thus sufficient for solving the equivalence problem.  
\end{theorem}

\subsection{Degenerate 4D Kundt spacetimes} \label{S47}
The first-order invariant $I_1$ is the same as in the previous section. In total, there are $\HD_2^4=12$ algebraically independent invariants of second order. We write down two of them:
\begin{gather*}
I_{2a} = H_{vv}, \qquad I_{2b}=\frac{((W_1)_{x^2v}-(W_2)_{x^1v})^2}{h_{11} h_{22}-h_{12}^2}.
\end{gather*}

Let us find an invariant horizontal frame. The horizontal 1-forms $\hat dI_1, \hat d I_{2a}, \hat d I_{2b}$ are independent: $\hat dI_1 \wedge \hat d I_{2a} \wedge \hat d I_{2b} \not\equiv 0$. Since $\D_v(I_1)=\D_v(I_{2a})=\D_v(I_{2b})=0$, the 1-forms have no $dv$-component. By solving the equations 
\[(\hat dI_1+a_1 \hat d I_{2a} + a_2 \hat d I_{2b}) (\D_{x^1})=0, \qquad (\hat dI_1+a_1 \hat d I_{2a} + a_2 \hat d I_{2b})(\D_{x^2})=0\]
for $a_1$ and $a_2$, we obtain an invariant 1-form which is proportional to $du$. We turn it into a horizontal vector field by using $g$, and denote the resulting invariant derivation, which is proportional to $\D_v$, by $\nabla_1$. Next, we define 
\[\nabla_2 = g^{-1} \hat d I_{2a}, \qquad \nabla_3 = g^{-1} \hat d I_{2b}. \] 
We complete the invariant horizontal frame by requiring $\nabla_4$ to satisfy 
\[ g(\nabla_1,\nabla_4)=1, \qquad g(\nabla_2,\nabla_4)=0, \qquad g(\nabla_3,\nabla_4)=0,\qquad  g(\nabla_4,\nabla_4)=0.\]

\begin{proposition}
The derivations $\nabla_1,\nabla_2,\nabla_3,\nabla_4$ are invariant, and independent on a Zariski open subset of $\ED^3$. 
\end{proposition}

We have $\D_v(g(\nabla_i,\nabla_j))\equiv 0$ for every $i$ and $j$. We choose an invariant for which this is not the case from the commutation relations $[\nabla_i,\nabla_j]=c_{ij}^k \nabla_k$. For instance, we have $\D_v(c_{23}^1) \not\equiv 0$.

\begin{theorem}
The differential invariants $I_1, I_{2a}, I_{2b}, c_{23}^1$ are horizontally independent, and thus sufficient for solving the equivalence problem.  
\end{theorem} 

%We have the following two invariant derivations.
%\begin{align*}
%\nabla_1 &= 2 \frac{((W_1)_v h_{22}-(W_2)_v h_{12}) D_{x^1} + ((W_2)_v h_{11}-(W_1)_v h_{12}) D_{x^2}}{h_{11} h_{22}-h_{12}^2} + (\cdots) D_v \\
%\nabla_2 &= \frac{\left((W_2)_v D_{x^1}-(W_1)_v D_{x^2}\right)}{(W_1)_{x^2v}-(W_2)_{x^1v}} +(\cdots) D_v
%\end{align*}

%Let us introduce the notation 
%\[ A=H_{x^1 vv} h_{22}-H_{x^2 vv} h_{12}, \qquad B= H_{x^1 vv} h_{12}-H_{x^2 vv} h_{11}, \]
%and 
%\begin{align*}
%T= (W_1 h_{22}-W_2 h_{12}) H_{x^1 vv}- (W_1 h_{12}-W_2 h_{11}) H_{x^2 vv}-2(h_{11} h_{22}-h_{12}^2)H_{uvv}.
%\end{align*}
%\[ \nabla_3 = \frac{A \D_{x^1}-B \D_{x^2}-T \D_v}{h_{11} h_{22}-h_{12}^2}\]

\section{Conclusion}

 % In this paper 
We considered the equivalence problem
for general and degenerate Kundt metrics with respect to the 
action of the pseudogroup of local diffeomorphisms. 
Denoting these classes of spacetimes by $\mathcal{K}$ and $\tilde{\mathcal{K}}$, respectively, 
we have (many other important subclasses are omitted):
 $$
\mathcal{K}\supset\tilde{\mathcal{K}}\supset\op{VSI}\supset\text{Kundt
waves}.
 $$
For general Kundt metrics the problem can be solved using
scalar polynomial invariants, but even then it is a nontrivial 
task to specify the required invariants, cf.\ \cite{ZM}.
For degenerate Kundt metrics, the spi are insufficient for separating metrics.

We use instead rational differential invariants, which separate jets of metrics in
general position within the class of degenerate Kundt metrics. By integrating the foliations $(\lambda,\Lambda)$ internal to the
class of Kundt metrics, one can normalize the set of admissible coordinates
and reduce the pseudogroup $\op{Diff}_{\text{loc}}(M)$ to $\Sym$ consisting of transformations that preserve the form of Kundt metrics expressed in terms of admissible coordinates.
The equivalence classes of Kundt metrics $\mathcal{K}$ 
(respectively degenerate Kundt metrics $\tilde{\mathcal{K}}$) 
with respect to all transformations are in bijective correspondence to those of
form \eqref{Kundt} with respect to the shape-preserving transformations:
 $$
\mathcal{K}/\op{Diff}_{\text{loc}}(M)=\E/\Sym\quad \text{ and }\quad
\tilde{\mathcal{K}}/\op{Diff}_{\text{loc}}(M)=\ED/\Sym.
 $$
In order to be consistent, in these equalities we should interpret $\mathcal{K}$ and $\tilde{\mathcal K}$ to mean the corresponding spaces of jets of metrics. The algebras of differential invariants 
consist of functions on those spaces.

Since our invariants are rational functions in jet-variables of low order and polynomial in higher jet-variables, there is a Zariski closed subset of jets of (degenerate) Kundt spacetimes that are not separated by the invariants we have found. By restricting to this Zariski closed set, and considering the Lie pseudogroup action on this set, it is possible to repeat the procedure and find an algebra of rational invariants separating generic jets of metrics in this singular set, etc. 

One should note that the coordinates used to create 
the signature variety $\Psi(X)$ need not be adapted to
$(\lambda,\Lambda)$. For instance, none of the invariants $I_i$ constructed
in Section \ref{S41} were required to be constant along $\Lambda$.
This however does not obstruct to solve the equivalence problem:
the foliation $\lambda$ is reconstructed from the first $(n-1)$ 
differential invariants and since the metric $g$ is determined, 
$\Lambda=\lambda^\perp$ is recovered. 
In principle, the Cartan invariants can be used for the same purposes,
yet with the formalism for differential invariants we have a better
control over the analytic properties of the functions in the algebra $\mathcal{A}$ of differential invariants.

\medskip

Several classes of transformations were considered in the literature
that are natural subgroups of $\Sym$. Reference \cite{AA} studied nil-Killing
fields defined as those vector fields $X$ on $M$ that are aligned
with respect to $\lambda$ and $\Ll_Xg$ is nilpotent wrt the filtration 
$(\lambda,\Lambda)$. It was shown in \cite{MA} that nil-Killing 
vector fields wrt $\lambda$, preserving $\lambda$, form a Lie algebra:
 $$
\g_\lambda=\{X\,:\,\Ll_X\lambda=\lambda\text{ and }
\Ll_Xg\text{ is of type III wrt }\lambda\}. 
 $$
This is an infinite-dimensional Lie subalgebra of $\sym$ given by \eqref{fg}.
The corresponding Lie pseudogroup of $\lambda$-aligned
transformations, preserving spi, depends on 1 function of $(n-1)$ arguments
(and other functions of fewer arguments), cf.\ \cite[Proposition 6]{AA}.
In fact, this pseudogroup consists of transformations \eqref{fG}
forming $\Sym$ such that the induced transformation of $(\bar{M}_u,h)$
is a $u$-parametric isometry.
 % The equivalence problem wrt $\g_\lambda$ has its own flavor 
 % (Kundt structures) but we will not consider it here.

A proper subalgebra of $\g_\lambda$ is the Lie algebra of Kerr-Schild 
vector fields wrt $\lambda$, defined as those $X$, preserving $\lambda$,
for which $\Ll_Xg\in S^2\lambda^*$ (has type N), see \cite{CHS}. 
This Lie algebra may be trivial, however if the 1-form $w=(W_i)_vdx^i$ 
on $\Lambda^*$ (important in our computations of 
invariants, see Proposition \ref{propI1}) is exact, then any infinitesimal
transformation $b(u)\p_v$ is a Kerr-Schild vector field, so this
algebra may also be infinite-dimensional. 

The equivalence problem of classes of spacetimes wrt to those and other 
Lie sub-pseudogroups may be of interest in its own right.

\subsection*{Acknowledgements} E. Schneider acknowledges full support via the Czech Science Foundation (GA\v{C}R no. 19-14466Y). This work was also partially supported by the project Pure Mathematics in Norway, funded by Trond Mohn Foundation and Tromsø Research Foundation.

\appendix
\section{Relative differential invariants}\label{ApA}

Here we demonstrate that the class of degenerate Kundt spacetimes is singled out
among all Kundt metrics by a relative invariant condition, so that the singular behavior
can be observed by studying orbits of the diffeomorphism pseudogroup on the jets of metrics.
Note that due to normalization of the vector $\ell$ as in \eqref{Kundt} the diffeomorphism pseudogroup
shrinks to the pseudogroup of shape-preserving transformations.

A function $f\in C^\infty(J^k\pi)$ is a relative differential invariant wrt a pseudogroup $\Sym$ if
$\varphi^*f=s_\varphi\cdot f$ $\forall\varphi\in\Sym$ for some nonzero function $s_\varphi$ on $J^k \pi$.
For the corresponding Lie algebra $\sym$ this translates into:
 $$
\Ll_{\xi^{(k)}}f=\omega(\xi)f\quad \forall\xi\in\sym
 $$
for some $\omega\in\sym^*\otimes C^\infty(J^k\pi)$. This $\omega$ is a 1-cocycle,
i.e.\ it satisfies the equation (cf.\ \cite{O})
 $$
\Ll_{\xi^{(k)}}\omega(\eta)-\Ll_{\eta^{(k)}}\omega(\xi)-\omega([\xi,\eta])=0\quad \forall \xi,\eta\in\sym.
 $$
Cocycles of the type $\omega=dh$ are called trivial. Cocycles modulo trivial ones are called
cohomology; in arbitrary order they form the group $H^1(\sym,C^\infty(J^\infty\pi))$.
To catch the algebraic structure of the jet-fibers we consider only such
cocycles that (modulo trivial) are polynomial in the jet-variables.
Such relative invariants are not plentiful.

Of course, absolute differential invariants are relative.
For any relative differential invariant $f$ the equation given by $f=0$ is $\Sym$-invariant.
If the invariant is genuinely relative (not absolute), then $\{f=0\}$ contains singular orbits.
Recall that an orbit is regular if a neighborhood of it is fibred by orbits, and it is called
singular otherwise.

Let us focus on the 3D case (as before we omit indices for $W_i$ and $h_{ij}$ here).
%We even start with a general spacetime possessing orthogonally integrable null congruence.
%For some functions $h,W,H$ this has the form
% \begin{equation}
%g= du \left(dv+H\,du+W\,dx\right) + h\,dx\,dx.
% \end{equation}
Since the action of $\Sym$ is transitive on $J^0\pi$,
as in Section \ref{S32}, we can translate any point to
 $$
p_0=\{u=0,x=0,v=0,h=1,H=0,W=0\}.
 $$

%The action of $\Sym_0^{(1)}$ on the 1-jet fiber $J^1_0\pi$ over $p_0$ is not transitive.
%It has precisely one relative invariant $h_v$. This can be straightforwardly
%computed from the cocycle equation. Thus the first (Kundt)
%condition\footnote{This is equivalent to geodesic, expansion-free and shear-free conditions.}
%$h_v=0$ comes as the first orbit singularity.

Consider the action of $\Sym_0^{(1)}$ on the fiber $\pi_{1,0}^{-1}(p_0) \cap \E_1$. Here $W_v^2$ is an absolute invariant, with the action transitive on its level sets\footnote{Note that $W_v^2$ is an absolute invariant only with respect to the stabilizer subgroup of the point $p_0$.
If we restore the entire group action, then both $W_v^2$ and $h$ are
relative invariants of the same weight, so that their ratio $I_1$ is an
absolute invariant.}.  As shown in Section \ref{S32}, we can bring any point to the point (omitting equations of $p_0$)
 $$
p_1=\{h_u=0,h_x=0,H_u=0,H_x=0,H_v=0,W_u=0,W_x=0,W_v=c\}.
 $$
Here $c \geq 0$ is the level parameter, and $h_u,h_x,\dots$ are jet-variables. Note that the first seven of these are normalized by translations, after which we are left with the action of $O(n-2)=O(1)= \mathbb Z_2$ on $W_v$.

%Assuming from now on $h_v=0$ (plus its differential corollaries) all orbits become
%regular.   Thus we fix a point on
%1-jet space $\E^1_0$ over $p_0$ so (omitting equations of $p_0$):
% $$
%p_1=\{h_u=0,h_x=0,H_u=0,H_x=0,H_v=0,W_u=0,W_x=0,W_v^2=c^2\}.
% $$

Next consider the action of the stabilizer pseudogroup $\Sym_1^{(2)}$ on
2-jets $\E^2_1=\E^2\cap\pi_{2,1}^{-1}(p_1)$. This space has dimension 15,
while the group acting on it has dimension 11. Thus we get 4 absolute
invariants, as established in Section \ref{S33} and explicitly given in Section \ref{S44}.
To get more precise structure of the orbit space note that the
group consists of 9 translations and 2 affine transformations.
The translations form a 9-dimensional Abelian group $\mathfrak{A}$ with the
Lie algebra
 $$
\p_{h_{xx}},\ \p_{h_{ux}},\ \p_{h_{uu}}+\p_{W_{ux}},\ \p_{H_{xx}}+\p_{W_{ux}},\
\p_{H_{ux}},\ \p_{H_{uu}},\ \p_{H_{uv}},\ \p_{W_{xx}},\ \p_{W_{uu}}.
 $$
We use them to set $h_{xx}=h_{ux}=h_{uu}=H_{xx}=H_{ux}=H_{uu}=H_{uv}=W_{xx}=W_{uu}=0$.
This global transversal to the action of $\mathfrak{A}$
can be identified with the quotient space $Q^6=\E^2_1/\mathfrak{A}$.
Let us introduce the coordinates
$z_1=-\frac12 W_{ux}$, $z_2=-W_{uv}$, $z_3=W_{xv}$,
$z_4=2W_{vv}$, $z_5=\frac13 W_{uv}-\frac23 H_{xv}$, $z_6=\frac23 W_{xv}+\frac43 H_{vv}$ on $Q$. Then
the infinitesimal affine transformations are
 $$
V_1=2z_1\p_{z_1}+z_2\p_{z_2}-z_4\p_{z_4}+z_5\p_{z_5},\quad
V_2= z_2\p_{z_1}+z_3\p_{z_2}+z_4\p_{z_3}+(z_3-z_6)\p_{z_5}.
 $$
They form a 2-dimensional solvable Lie algebra with $[V_1,V_2]=-V_2$.
Since  $V_2$ is nilpotent, any polynomial relative invariant must belong to its kernel. The linear polynomials in the kernel are spanned by $z_4$ and $z_6$. Here $z_6$ is an absolute invariant, while $z_4$ has weight $-1$ with respect to $V_1$. The relative invariant $z_4$ gives the first condition for degenerate Kundt spacetimes: $W_{vv}=0$. 

By extending this analysis to $\E^3$, we see that the function $H_{vvv}$ is not a relative invariant, but becomes so when we restrict to the subset in $\E^3$ given by $W_{vv}=0$ and its differential consequences. This (conditional) relative invariant also has weight $-1$ with respect to (the prolongation of) $V_1$. 

Note that there exist other nontrivial relative invariants, of higher degree. The invariant $W_{vv}$ on $\E^2$ is singled out by having negative weight; all other relative invariants with negative weight have $W_{vv}$ as a factor. The invariant $H_{vvv}$ on the sub-PDE given by $W_{vv}=0$ is determined uniquely in the same way. Thus the degenerate Kundt conditions $W_{vv}=0, H_{vvv}=0$ arise from investigations of singularities of the $\Sym$ action on $\E^2$ and $\E^3$.

 \begin{theorem}
The function $W_{vv}$ is a relative invariant of the $\Sym$ action on $\E^2$. On the submanifold in $\E^3$ given by $W_{vv}=0$ and its differential consequences, the function $H_{vvv}$ is a relative invariant of the $\Sym$ action. 
 \end{theorem}

%Similarly, by considering 3-jets we obtain that $W_{vvv}$ is a relative invariant but
%its weight is a combination of those for $W_v,W_{vv}$.
 % we get an absolute invariant with the symbol $\frac{W_vW_{vvv}}{W_{vv}^2}$ but it vanishes for degenerate Kundt
%In addition, the orbit space analysis shows that $H_{vvv}$ is not a relative invariant,
%but it becomes such modulo the condition $W_{vv}=0$.
%Thus we recover the degeneracy conditions for Kundt spacetimes from the analysis of singularities for the $\Sym$ action on $\E$.

The same idea can be applied in higher dimensions. The Lie algebra spanned by $V_1$ and $V_2$ is then replaced by an $(n-1)$-dimensional Lie algebra from which information about singular orbits can be read. In this case, $W_{vv}$ should be considered as a tensorial relative invariant (covector), whose corresponding zero-set has codimension greater than $1$. Then the theorem holds true in higher dimensions as well.

%Thus the degenerate Kundt conditions $W_{vv}=0,H_{vvv}=0$ are $\Sym$-invariant and
%they naturally arise from investigations of singularities of orbits in jets of metrics.

%%%%%%%%%%%%%%%%%%%%%%%%%%%%%%%%%%%%%%%%%%%%%%%%%%%%%%%%%%%%%%%%%%%%%%%%%%


\begin{thebibliography}{11}
\footnotesize

\bibitem{AA}
M.\ Aadne, {\it Nil-Killing vector fields and type III deformations},
Journal of Mathematical Physics {\bf 61}, 122502 (2020).

\bibitem{C}
E.\,Cartan, {\it Le\c{c}on sur la g\'{e}ometrie des Espaces de Riemann},
2nd edn, Paris (1946).

\bibitem{CHP}
A. Coley, S. Hervik, N. Pelavas, {\it Spacetimes characterized by their scalar curvature invariants},
Class. Quant. Grav. {\bf 26}, 025013 (2009).

\bibitem{CHP2}
A. Coley, S. Hervik, N. Pelavas, {\it Lorentzian manifolds and scalar curvature invariants},
Class. Quantum Grav. {\bf 27}, 102001 (2010).

\bibitem{CHPP}
A. Coley,  S. Hervik,  G. Papadopoulos,  N. Pelavas, {\it Kundt  Spacetimes}, Class. Quant. Grav. {\bf 26}, 105016 (2009).

\bibitem{CHS}
B.\ Coll, S.\ Hildebrandt, J.\ Senovilla,
{\it Kerr-Schild symmetries}, Gen.\ Rel.\ Grav.\ {\bf 33}, 649-670 (2001).

\bibitem{CO}
S.\ Console, C.\ Olmos, {\it Curvature invariants, Killing vector fields,
connections and cohomogeneity}, Proc.\ AMS {\bf 137}, 1069--72 (2008).

\bibitem{Kar}
A. Karlhede, {\it A Review of the Geometrical Equivalence of Metrics in General Relativity}, Gen. Rel. Grav. {\bf 12}, 693 (1980).

\bibitem{Ker}
R.\ Kerr, {\it Scalar invariants and groups of motions in a $V_n$
with positive definite metric tensor}, Tensor {\bf 12}, 74--83 (1962).

\bibitem{Kou}
A.\ Koutras, C.\ McIntosh, {\it A metric with no symmetries or invariants},
Class.\ Quantum Grav.\ {\bf 13}, 4749 (1996).

\bibitem{ESEFE}
D. Kramer, H. Stephani, M. MacCallum, E. Herlt, {\it Exact Solutions of Einstein’s Field Equations}, Cambridge University Press (1980).

\bibitem{Kr}
B.\ Kruglikov, {\it Poincar\'e function for moduli of differential-geometric structures},
Moscow Math.\ Journ.\ {\bf 19}, no.\ 4, 761-788 (2019).

\bibitem{KL1}
B.\ Kruglikov, V.\ Lychagin, {\it Geometry of Differential equations\/}, Handbook of Global Analysis, Ed. D.Krupka, D.Saunders, Elsevier, 725-772 (2008).

\bibitem{KL2}
B.\ Kruglikov, V.\ Lychagin, {\it Global Lie-Tresse theorem}, Selecta Mathematica {\bf 22}, 1357-1411 (2016).

\bibitem{KMS}
B. Kruglikov, D.\ McNutt, E.\ Schneider,
{\it Differential invariants of Kundt waves},
Class.\ Quantum Grav.\ {\bf 36}, 155011 (2019).

\bibitem{KT}
B.\ Kruglikov, K.\ Tomoda, {\it A criterion for the existence of
Killing vectors in 3D}, Class.\ Quantum Grav.\ {\bf 35}, 165005 (2018).

\bibitem{Ku}
W.\ Kundt, {\it The plane-fronted gravitational waves},
Z.\ Phys. {\bf 163}, 77 (1961).

\bibitem{LY}
V. Lychagin, V. Yumaguzhin, {\it Invariants in Relativity Theory\/}, Lobachevskii Journal of Mathematics {\bf 36}, 298-312 (2015).

\bibitem{MA}
D.\ McNutt, M.\ Aadne, {\it I-preserving diffeomorphisms of Lorentzian
manifolds}, Journ.\ Math.\ Phys.\ {\bf 60}, no.\ 3, 032501 (2019).

\bibitem{MMC}
D. McNutt, R. Milson, A. Coley, {\it Vacuum Kundt waves}, Class. Quant. Grav. {\bf 30}, 055010 (2013).

\bibitem{O}
P. Olver, {\it Equivalence, Invariants and Symmetry}, Cambridge University Press, Cambridge (1995).

\bibitem{OP}
P. Olver, J. Pohjanpelto, {\it Differential invariant algebras of Lie pseudo-groups\/},
Adv. Math. {\bf 222}, no. 5, 1746--1792 (2009).

\bibitem{PR}
R. Penrose, W. Rindler, {\it Spinors and Spacetime Vol. 1}, Cambridge University Press (1984).

\bibitem{PPCM}
V. Pravda, A. Pravdova, A. Coley, R. Milson, {\it All spacetimes with vanishing curvature invariants},
Class. Quant. Grav. {\bf 19}, 6213 (2002).

\bibitem{S}
I.\ Singer {\it Infinitesimally homogeneous spaces},
Comm.\ Pure Appl.\ Math.\ {\bf 13} 685--97 (1960).

\bibitem{T}
T.Y.\ Thomas, {\it The Differential Invariants of Generalized Spaces},
Cambridge University Press (1934).

\bibitem{W}
H.\ Weyl, {\it The classical Groups: Their Invariants and Representations},
Princenton University Press (1946).

\bibitem{ZM}
E.\ Zakhary, C.\,B.\,G.\ Mcintosh, {\it A complete set of Riemann
invariants}, Gen. Rel. Grav. {\bf 29}, 539--581 (1997).

\end{thebibliography}
\end{document}